\documentclass{amsart}
\usepackage[utf8]{inputenc}
\usepackage{amsmath}
 \usepackage{amsfonts}
 \usepackage{stmaryrd}
\usepackage{amssymb,latexsym}
\usepackage{amsfonts,mathrsfs}
\usepackage[margin=1.8in]{geometry}
\usepackage[all,cmtip]{xy}
\usepackage[colorlinks= true,pdfborder={0 0 0}]{hyperref}
\usepackage{graphicx}
 \usepackage{stmaryrd}
 \usepackage{esint}
 \usepackage {verbatim}
\usepackage{color}
\allowdisplaybreaks
\def\pd#1{\dfrac{\partial}{\partial #1}}
\def\f#1#2{\frac{#1}{#2}}

\def\pa{\partial}
\def\n{\nabla}

\def\a{\alpha}
\def\b{\beta}

\def\({\left (}
\def\){\right )}
\def\<{\langle}
\def\>{\rangle}

\newcommand{\bel}[1]{\begin{equation}\label{#1}}
\newcommand{\lab}[1]{\label{#1}}
\newcommand{\be}{\begin{equation}}
  \newcommand{\beq}{\begin{equation}}

\newcommand{\ba}{\begin{eqnarray}}
\newcommand{\ea}{\end{eqnarray}}
\newcommand{\rf}[1]{(\ref{#1})}
\newcommand{\qe}{\end{equation}}
\newcommand{\eeq}{\end{equation}}

\newtheorem{thm}{Theorem}[section]
\newtheorem{cor}[thm]{Corollary}
\newtheorem{lem}[thm]{Lemma}
\newtheorem{prop}[thm]{Proposition}
\newtheorem{defn}[thm]{Definition}
\newtheorem{rem}[thm]{Remark}
\newtheorem{claim}{Claim}[section]
\newtheorem*{ack}{Acknowledgement}

\newcommand{\norm}[1]{\left\Vert#1\right\Vert}
\newcommand{\abs}[1]{\left\vert#1\right\vert}
\newcommand{\set}[1]{\left\{#1\right\}}

\newcommand{\R}{\mathbb R}
\newcommand{\eps}{\varepsilon}

\newcommand{\pfrac}[2]{\frac{\partial #1}{\partial #2}}

\newcommand{\Om}{\Omega}
\newcommand{\om}{\omega}

\newcommand{\Sgm}{\Sigma}
\newcommand{\sub}{\subset}

\newcommand{\red}[1]{ {\color{red} #1} }

\newcommand{\tr}{\mbox{\rm tr\,}}

\title[Stability of VPMCF and APMCF in
Schwarzschild space]{Stability of volume and area preserving mean curvature flows in the
Schwarzschild space}
\author{Yaoting Gui, Yuqiao Li, Jun Sun}
\address{Yaoting Gui, School of Mathematical Sciences, Xiamen University, Xiamen, 361005, P.R.China.;}
\address{
Yuqiao Li, Department of Mathematics, Hefei University of Technology, Hefei, 230009, P.R.China.}
\address{
Jun Sun, School of Mathematics and Statistics, Wuhan University, Wuhan, 430072, P.R.China.}
\email{ytgui@bicmr.pku.edu.cn, lyq112@mail.ustc.edu.cn, sunjun@whu.edu.cn}

\begin{document}

\numberwithin{equation}{section}

\begin{abstract}
    In this paper, we investigate the stability of the volume preserving mean curvature flow (VPMCF) and area preserving mean curvature flow (APMCF) in the Schwarzschild space. 
    We show that if the initial hypersurface is sufficiently close to a coordinate sphere, these flows exist globally in time and converge smoothly to a constant mean curvature (CMC) hypersurface, namely a coordinate sphere. For asymptotically Schwarzschild spaces, if the initial hypersurface is sufficiently close to an isoperimetric hypersurface, outside of some large compact set, the flow still exists for all time and converges to a CMC hypersurface exponentially fast.
\end{abstract}

\maketitle

\section{ Introduction}
\allowdisplaybreaks

Extrinsic curvature flow has proven to be a powerful tool for exploring geometric and topological properties of submanifolds. For example, Huisken pioneered the study of mean curvature flow \cite{MR0837523} and its volume preserving variant in Euclidean space \cite{MR0921165}, demonstrating that convex hypersurfaces are always diffeomorphic to spheres. As a corollary of the convergence of volume preserving flow, he showed convex domains satisfy the isoperimetric inequality in \cite{MR0921165}. Earlier results concerning planar curves have been investigated by Gage-Hamilton (\cite{MR840401}), Grayson (\cite{MR906392}) and Gage (\cite{MR848933}).

Let us first set up the basic background. Let $M^n$ be a closed manifold and $N^{n+1}$ be a Riemannian manifold. For a given submanifold $F_0:M^n\to N^{n+1}$, the \textbf{volume preserving mean curvature flow (VPMCF)} is a family of maps $F_t=F(\cdot,t):M^n\to N^{n+1}$ evolving by
\begin{equation}\label{e-VPMCF}
\begin{cases}   
    \frac{\partial F}{\partial t}(x, t)&=[h(t)-H(x,t)]\nu(x,t),\\
    F(\cdot,0)&=F_0,
\end{cases}
\end{equation}
where $H$ is the mean curvature of $M_t=F_t(M),\ \nu$ is the outward unit normal to $\Sigma_t$, and
\begin{equation*}
h(t)=\frac{\int_{M_t}Hd\mu_t}{\int_{M_t}d\mu_t}
\end{equation*}
is the average of the mean curvature. 

Along the VPMCF, the volume of the region $\Omega_t$ enclosed by $M_t$ remains constant, while the area of $M_t$ is nonincreasing (see (\ref{e-VPMCF-vol}) and (\ref{e-VPMCF-area})). As a consequence, the application of the VPMCF leads to the isoperimetric inequality, which relates the area of a hypersurface to its enclosed volume. If the VPMCF exists globally and converges smoothly as $t\to\infty$, the limiting hypersurface must have constant mean curvature (CMC). This approach was employed by Huisken-Yau \cite{huisken1996definition} to construct the constant mean curvature surface foliation in the asymptotic Schwarzschild space. They showed that starting from a large coordinate sphere ensures the global existence and convergence to a CMC surface exponentially fast.

Stability results for the VPMCF have also been established in space forms. Li (\cite{li2009volume}) proved the long-term existence and convergence in Euclidean space under the assumption that the integral of the  traceless second fundamental form is small, similar results are obtained in other space forms (\cite{XuHW2014spaceform}, \cite{MR4762133}, \cite{MR4910985}). 
Additionally, Miglioranza studied the VPMCF on compact manifolds, assuming the initial surface lies within a small geodesic ball \cite{Miglioranza2020TheVP}. See also Freire-Alikakos \cite{alikakos2003normalized} and \cite{escher1998volume}. 

\vspace{.1in}

In this paper,  we mainly focus on the volume preserving mean curvature flow (VPMCF) in Schwarzschild space as a typical example; see \cite{escher1998volume} for similar results in the Euclidean space. Recall that an $(n+1)$-dimensional Riemannian manifold $(N^{n+1},\bar{g})$ is referred to as the \textbf{Schwarzschild space of mass $m$} if $N$ is diffeomorphic to ${\mathbb R}^{n+1}\setminus B_1$ and 
\begin{equation*}
    \bar{g}=\phi^{\f4{n-1}}\delta,
\end{equation*}
where $\phi(x)=1+\f m{2r^{n-1}}$, $r=\abs{x}_{\delta}>(\f m2)^{\f1{n-1}}$, and $m$ is a positive constant called the mass. The hypersurface $\Sigma_0:r=(\f m2)^{\f1{n-1}}$ is called the horizon. It models the black hole in physical world. We remark that the Schwarzschild space refered  here is usually called half Schwarzschild space compared with the doubly Schwarzschild space, which is a double of our half Schwarzschild space, across the boundary. Here because the uniqueness result are not true any more in doubly Schwarzschild space\cite{Brendle2013constant}, we should work with the Schwarzschild space defined here. 

Before stating the main thoerem, we need the following definition.
\begin{defn}
    A hypersurface $M$ is said to be homologous to the horizon if the enclosed domain $\Omega$ is bounded by $M$ and $\Sigma_0$.
\end{defn}
Our first main result establishes the long-time existence and convergence of the VPMCF in Schwarzschild space for initial hypersurfaces sufficiently close to any coordinate sphere. 
\begin{thm}\label{thm1.1}
    Let $(N^{n+1},\bar{g})$ be the Schwarzschild space of mass $m>0$, and let $F:M\times I\rightarrow N$ denote the VPMCF starting from an initial hypersurface $M_0$ that is homologous to the horizon $\Sigma_0$. Let $\Om_0$ be the region enclosed by $M_0$ and $\Sigma_0$.
   There exists a constant $\delta>0$ such that if $M_0$ is $\delta$-close in $C^2$-norm to a coordinate sphere $\Sigma_{r_0}$, and the volume of $\Omega_0$ equals the volume of the domain enclosed by $\Sigma_{r_0}$ and $\Sigma_{0}$, then the flow exists globally and converges smoothly to $\Sigma_{r_0}$. 
\end{thm}

\begin{rem}
    This approach can be extended to general warped product spaces under mild assumptions \cite{Brendle2013constant},  as it depends only on the uniqueness of constant mean curvature (CMC) hypersurfaces and curvature estimates in the ambient space.
\end{rem}

\begin{rem}
Unlike Huisken-Yau’s result (\cite{huisken1996definition}), no large-radius assumption for the initial sphere is required.
\end{rem}

An alternative approach to the isoperimetric inequality involves the area-preserving mean curvature flow, where $F_t=F(\cdot,t):M^n\to N^{n+1}$ evolves according to
\begin{equation}\label{e-APMCF}
\begin{cases}   
    \frac{\partial F}{\partial t}(x, t)&=[h_0(t)-H(x,t)]\nu(x,t),\\
    F(\cdot,0)&=F_0,
\end{cases}
\end{equation}
where
\begin{equation*}
h_0(t)=\frac{\int_{M_t}H^2d\mu_t}{\int_{M_t}Hd\mu_t}.
\end{equation*}
For a closed manifold $M^n$ and a Riemannian manifold $N^{n+1}$, the \textbf{area-preserving mean curvature flow (APMCF)} evolving from a mean convex submanifold $F_0:M^n\to N^{n+1}$ preserves the area of $M_t$ while the enclosed volume $\Omega_t$ increases over time. If the isoperimetric inequality holds, this yields a uniform upper bound on $\text{vol}(\Omega_t)$, enabling proofs of the inequality via the APMCF.

McCoy proved that if the initial hypersurface is uniformly convex in ${\mathbb R}^{n+1}$, the APMCF exists globally and converges to a round sphere as $t\to\infty$ (\cite{MR2031450}). Recently, we established stability results for the APMCF in asymptotically Schwarzschild space (\cite{Gui-Li-Sun}). This stability says that the flow will finally converge to a coordinate sphere when the $L^2$ norm of the traceless second fundamental form is sufficiently small. Furthermore, we reconstruct the existence of Huisken-Yau's CMC foliation in the asymptotic Schwarzschild space using the APMCF, replacing the VPMCF employed by Huisken and Yau. 

Analogous to Theorem \ref{thm1.1}, another main result of this paper establishes the long-time existence and convergence of the APMCF in the Schwarzschild space of mass $m>0$ when the initial hypersurface is still sufficiently close to a coordinate sphere. Specifically,
\begin{thm}\label{thm1.1-2}
    Let $(N^{n+1},\bar{g})$ be the Schwarzschild space of mass $m>0$, and let $F:M\times I\rightarrow N$ be an APMCF starting from an initial hypersurface $M_0$ which is homologous to the horizon $\Sigma_0$, a coordinate sphere of radius ${(\frac{m}{2})}^{\f1{n-1}}$. 
   There exists a constant $\delta>0$ such that if $M_0$ is $\delta$-close in $C^2$-norm to a coordinate sphere $\Sigma_{r_0}$ with the area of $M_0$ equals the area of $\Sigma_{r_0}$, then the flow exists globally and converges smoothly to $\Sigma_{r_0}$. 
\end{thm}

\vspace{.1in}

We now briefly outline the proofs of the main theorems. We  provide a unified proof for the long-time existence of both the VPMCF and APMCF. Our method builds on Miglioranza’s thesis, where he proved the long-time existence and subconvergence of the VPMCF in small geodesic balls within general Riemannian manifolds \cite{Miglioranza2020TheVP} and also borrow some idea from minimal surface compactness. The proof we employ here primarily leverages the isoperimetric property and the volume- or area-preserving nature of the flows, showing that under isoperimetric ratio conditions, these quantities remain close to those of coordinate spheres throughout the flow. This ensures that the principal curvatures stay near those of a coordinate sphere, and a perturbation theorem guarantees the limiting hypersurface's closeness to the coordinate sphere. Here, the perturbation theorem states the following: 
\begin{thm}\lab{thm1.3}
    Let  $\Om_k\Subset \Om\sub N$ be a sequence of bounded domains in the Schwarzschild space, where $\Omega$ contains the horizon, and let  $M_k\subset\pa \Om_k$ be smooth boundary components homologous to the horizon (in particular, $\pa\Om_k$ is embedding). Let $\Om_0$ denote the domain enclosed by the horizon and some coordinate sphere such that there exist two positive constants $C_i, i=1,2$, for all $k$, satisfying
    \begin{enumerate}
        \item $\abs{\Om_k}=\abs{\Om_0}, \abs{\pa\Om_k}\leq C_1$;
        \item $\abs{A}_{\pa\Om_k}\leq C_2, \abs{\n A}_{\pa\Om_k}\leq C_2$;
        \item $\f{\abs{M_k}^{n+1}}{|{\Om_k}|^n}\leq I_S(|\Omega_0|)+\rho_k$, with $\rho_k\rightarrow0$ as $k\to \infty$;
    \end{enumerate}
    where $|\cdot|$ represents the volume of a domain or the area of a hypersurface, $A$ is the second fundamental form, and $I_S(|\Omega_0|)$ is the isoperimetric ratio of the Schwarzschild space defined by
    \begin{small}
   \begin{align*}
       &I_S(|\Omega_0|):=\\
       &\inf_{\text{$\Sigma$ homologous to $\Sigma_0$}}\set{\f{\abs{\Sigma}^{n+1}}{\abs{\Om_0}^n}:\text{the region enclosed by $\Sigma$ and $\Sigma_0$ has the same volume as $\Om_0$}}.
   \end{align*}
\end{small}
Then, up to a subsequence, $M_k$ converges to some coordinate sphere $M_\infty$. Let $\Om_\infty$ denote the domain enclosed by $M_\infty$ and the horizon; we have $\abs{\Om_\infty}=\abs{\Om_0}$. 
\end{thm}

The proof adapts techniques from minimal surface compactness, with control of the isoperimetric ratio eliminating limits of higher multiplicity. 

It should be remarked that the idea using isoperimetric ratio can be employed to establish a general long time existence and convergence result in an asymptotic Schwarzschild manifold. We notice that Eichmair and Metzger (\cite{EichmairInventMath2013}) proved that if $(M^n,g)$ is $C^2$ asymptotic to the Schwarzschild manifold of mass $m>0$, see the definition in Section 4, then there exists $V_0>0$ such that for all $V\geq V_0$, there is a unique isoperimetric hypersurface $\Sigma_V$ enclosed a domain with volume $V$. This isoperimetric hypersurface is strictly stable with constant mean curvature(a CMC hypersurface) and stay close to a coordinate sphere. These isoperimetric hypersurfaces form a foliation of $M\backslash \Omega_0$. The next theorem shows that the volume and area preserving mean curvature flows will converge to the isoperimetric hypersurface if the initial hypersurface is sufficiently close to it.

\begin{thm}
    Let $(N^{n+1}, \bar{g})$ be a $C^2$ asymptotically Schwarzschild space. There exists a constant $\delta > 0$ such that the following holds: if a hypersurface $M_0$ satisfies two conditions — first, the volume of the region enclosed by $M_0$ and the horizon satisfies $V \geq V_0$; second, $M_0$ is $\delta$-close to $\Sigma_V$ in the $C^2$ norm — then the volume- and area-preserving mean curvature flow exists for all time and converges to $\Sigma_V$ as $t \to \infty$.
\end{thm}

\begin{rem}
    It should be remarked that the existence result of CMC hypersurfaces holds true even in asymptotically flat spaces under the assumption that the initial hypersurface has a controlled second fundamental form together with its derivatives and the isoperimetric ratio. This can be archived by the usual minimizing method and the controlled isoperimetric ratio eliminate the high multiplicity case. This gives the existence of CMC hypersurface in asymptotic flat space.
\end{rem}
\vspace{.1in}

Finally, we briefly outline the sketch of this paper. Section 2 covers some necessary preliminaries and notations. The following section is devoted to giving unified proofs of the main theorems. The last section is to generalize the main theorems in asymptotically Schwarzschild spaces.

\begin{ack}
We want to express our sincere thanks to Prof. Jiayu Li for constant encouragement and guidance. The first author would like to thank Prof. Gang Tian for helpful discussion. Special thanks also give to Prof. Ze Li for bringing us detailed explanation of the center manifold analysis. The first author is partially supported by National Key RD Program of China 2020YFA0712800 and Initial Funding of No.X2450216. The second author is supported by Initial Scientific Research Fund of Young Teachers in Hefei University of Technology of No.JZ2024HGQA0122. The third author is supported by NSFC No. 12531002, No. 12271039 and No. 12071352.
\end{ack}
\vspace{.2in}

\section{Preliminaries}
We adopt the notations from Huisken-Yau in \cite{huisken1996definition} and its $(n+1)$-dimensional extensions in \cite{Miglioranza2020TheVP}.
Let $({\mathbb R}^{n+1}, \delta)$ be the $(n+1)$-dimensional Euclidean space with the flat metric $\delta$ and Euclidean coordinates $\{y_{\alpha}\}$, $\alpha=1,\cdots,n+1$. Greek indices range from 1 to $(n+1)$, Latin indices range from 1 to $n$ and we write 
$$
r^2=\sum_{\alpha=1}^{n+1}(y_{\alpha})^2
$$ 
for the Euclidean distance.
We consider the half Schwarzschild space $(N, \bar{g})$, which is conformal to $(\mathbb{R}^{n+1}\setminus B_1, \delta)$ with $\bar{g}=\phi^{\frac{4}{n-1}}\delta$, where $\phi=1+\frac{m}{2r^{n-1}}$ and $m>0$. 
Let $\bar{\n},\bar{R}m$ and $\bar{R}_{\alpha\beta}$ denote the covariant derivative, Riemann and Ricci curvatures 
with respect to $\bar{g}$.

Suppose $\Sigma$ is an $n$-dimensional hypersurface in $(N, \bar{g})$ with induced metric $g=\{g_{ij}\}$ on $\Sigma$.
Denote by $\nu,\n$ the unit outward normal and the covariant derivative on $\Sigma$, respectively. Let $\kappa_i$ be the principal curvatures of $\Sigma$ with corresponding unit principal vectors $f_i$.
In the orthonormal frame $f_1,\cdots, f_{n+1}=\nu$ on $N$, the second fundamental form $A=\{h_{ij}\}$ is given by
\[ h_{ij}=\<\bar{\n}_{f_i}\nu, f_j\>=\kappa_i\delta_{ij}. \]
We write the mean curvature $H$ of $\Sigma$ and the square of the norm of the second fundamental form by 
\[ H=g^{ij}h_{ij}=\sum_i\kappa_i, \quad \abs{A}^2=g^{ik}g^{jl}h_{ij}h_{kl}=\sum_i\kappa_i^2. \]
We will adopt $\mathring{A}$ to denote the frequently used geometric quantity, the traceless second fundamental form, i.e.
\[\mathring{h}_{ij}=h_{ij}-\f1{n} Hg_{ij},\]
The corresponding norm of $\mathring{A}$ is
\[ \mathring{\abs{A}}^2=g^{ik}g^{jl}\mathring{h}_{ij}\mathring{h}_{kl}=\abs{A}^2-\f1{n} H^2=\f1{n}\sum_{i\neq j}(\kappa_i-\kappa_j)^2. \]

We next calculate the evolution equations for VPMCF which also appeared in \cite{huisken1996definition}.

\begin{lem}(Lemma 3.6, 3.7, 3.8 in \cite{huisken1996definition} in dimension 3 and Lemma 9 in \cite{li2009volume})\label{evoeqns}
Along the VPMCF \eqref{e-VPMCF}, we have the evolution equations
    \begin{align*}
        \pfrac{} {t}g_{ij}=&2(h-H)h_{ij},\\
        \pfrac{} {t}d\mu_t=&H(h-H)dV_t,\\
        \pfrac{} {t}H=& \Delta H+(H-h) (|A|^2+\bar{R}ic(\nu, \nu)), \\
        \pfrac{}{t}h_{ij}=&\Delta h_{ij}-2Hh_{il}h_{lj}+hh_{il}h_{lj}+|A|^{2}h_{ij}-h\bar{R}_{\nu i\nu j}+h_{ij}\bar{R}_{\nu l \nu l}\\
        &-h_{jl}\bar{R}_{lmim}-h_{il}\bar{R}_{lmjm}+2h_{lm}\bar{R}_{limj}-\bar{\n}_{j}\bar{R}_{\nu lil}-\bar{\n}_{l}\bar{R}_{\nu ijl},\\
        \pfrac{} {t}|A|^2=&\Delta\abs{A}^2-2\abs{\n A}^2+2\abs{A}^4
    -2h\tr{A^3}+2\abs{A}^2\bar{R}ic(\nu,\nu)\\ 
    & -2hh_{ij}\bar{R}_{\nu i\nu j}-4(h_{ij}h_{jl}\bar{R}_{lmim}-h_{ij}h_{lm}\bar{R}_{limj})\\ 
    & -2h_{ij}(\bar{\n}_j\bar{R}_{\nu lil}+\bar{\n}_l\bar{R}_{\nu ijl}),
    \end{align*}  
where $d\mu_t$ is the area form of $M_t$.
\end{lem}

As a corollary, we have the variational formula for the volume and the area along the VPMCF.

\begin{cor} Along the VPMCF, we have
\begin{equation}\label{e-VPMCF-vol}
        \frac{d}{dt}{\rm Vol}(\Omega_t)=0,
\end{equation}
and
\begin{equation}\label{e-VPMCF-area}
        \frac{d}{dt}{\rm Area}(M_t)=-\int_{M_t}(h-H)^2d\mu_t.
\end{equation}
\end{cor}

For the APMCF, we have

\begin{lem}
Along the APMCF \eqref{e-APMCF}, we have the evolution equations
    \begin{align*}
        \pfrac{} {t}g_{ij}=&2(h_0-H)h_{ij},\\
        \pfrac{} {t}d\mu_t=&H(h_0-H)d\mu_t,\\
        \pfrac{} {t}H =&\Delta H+(H-h_0) (|A|^2+\bar{R}ic(\nu, \nu)), \\
        \pfrac{}{t}h_{ij}=&\Delta h_{ij}-2Hh_{il}h_{lj}+h_0h_{il}h_{lj}+|A|^{2}h_{ij}-h_0\bar{R}_{\nu i\nu j}+h_{ij}\bar{R}_{\nu l\nu l}\\
        &-h_{jl}\bar{R}_{lmim}-h_{il}\bar{R}_{lmjm}+2h_{lm}\bar{R}_{limj}-\bar{\n}_{j}\bar{R}_{\nu lil}-\bar{\n}_{l}\bar{R}_{\nu ijl},\\
        \pfrac{} {t}|A|^2=&\Delta\abs{A}^2-2\abs{\n A}^2+2\abs{A}^4
    -2h_0\tr{A^3}+2\abs{A}^2\bar{Ric}(\nu,\nu)\\ 
    & -2h_0h_{ij}\bar{R}_{\nu i\nu  j}-4(h_{ij}h_{jl}\bar{R}_{lmim}-h_{ij}h_{lm}\bar{R}_{limj})\\ 
    & -2h_{ij}(\bar{\n}_j\bar{R}_{\nu lil}+\bar{\n}_l\bar{R}_{\nu ijl}).
    \end{align*} 
\end{lem}

As a corollary, we have the variational formula for the volume and the area for the APMCF.

\begin{cor} Along the APMCF, we have
\begin{equation}\label{e-APMCF-vol}
        \frac{d}{dt}{\rm Vol}(\Omega_t)=\int_{M_t}(h_0-H)d\mu_t,
\end{equation}
and
\begin{equation}\label{e-APMCF-area}
        \frac{d}{dt}{\rm Area}(M_t)=0.
\end{equation}
\end{cor}
By the definition of $h_0$, the Hölder inequality and (\ref{e-APMCF-vol}), we see that
\begin{equation}\label{e-APMCF-vol-2}
        \frac{d}{dt}{\rm Vol}(\Omega_t)=\int_{M_t}(h_0-H)d\mu_t=\frac{\int_{M_t}H^2d\mu_t}{\int_{M_t}Hd\mu_t}|M_t|-\int_{M_t}Hd\mu_t\geq 0,
\end{equation}
provided that $\int_{M_t}Hd\mu_t>0$.

\vspace{.2in}

\section{Proof of the Main Theorems}

In this section, we provide  detailed proofs of the main theorems. The proof consists of two steps: the long-time existence and the convergence. We  first prove the corresponding results for the volume-preserving mean curvature flow.

\subsection{The long time existence of VPMCF}

Due to the parabolic nature, we may assume the flow exists for a short time \cite{alma991007096229706271}. By combining  curvature estimates with the Bonnet-Myers theorem, we  show that the flow remains bounded and its geometry can be controlled at the maximal time. Hence, the flow can be extended for a short time, contradicting the maximality of the time.  


 Let $\Om_0$ denote the domain enclosed by the surface $M_0$ and the horizon $\Sigma_0$ in the Schwarzschild space $(N, \bar{g})$. Given $\delta>0$ such that $M_0$ is $\delta$ close to the coordinate sphere $\Sigma(r_0)$ in $C^2$ norm.
The principal curvatures of the coordinate sphere $\Sgm(r)$ with $r=r_0$ can be explicitly calculated as
\[
\bar{\lambda}_i=\phi^{-\frac{n+1}{n-1}}r_0^{-1}\left(1-\frac{m}{2r_0^{n-1}}\right), 
\]
for all $i=1, 2,\cdots, n.$
By scaling, we may assume the principal curvatures of $\Sigma(r_0)$ are equal to 1. Since $M_0$ is $\delta$-close to $\Sgm(r_0)$, by modifying $\delta$ if necessary, we may assume the principal curvatures $\kappa_i, (i=1,2,\cdots, n)$ of $M_0$ satisfy
\[
\f12\leq \kappa_i(x,0)\leq2, \quad \forall i=1,\cdots,n,\quad\forall x\in M_0.
\]
Then the initial surface $M_0$ satisfies the following estimates
\[
     \abs{H}(x,0)=\sum_{i=1}^n\kappa_i\leq 2n,\quad \abs{A}(x,0)=\sqrt{\sum_{i=1}^n\kappa_i^2}\leq 2\sqrt{n}.
\]
Set $C'=5n$.
 We choose $\rho_0=\rho_0(n, \delta, C')$ such that 
\begin{equation}\label{iso}
\f{\abs{M_0}^{n+1}}{\abs{\Om_0}^n}\leq I_S(|\Omega_0|)+\rho_0,
\end{equation}
where $I_S(|\Omega_0|)$ is the isoperimetric ratio of the Schwarzschild space as defined in Theorem \ref{thm1.3}.

Condition \eqref{iso} implies  $\abs{M_0}\leq[(I_S(|\Omega_0|)+\rho_0)\abs{\Om_0}^n]^{\f1{n+1}}=\tilde{C}$.
Since the flow preserves volume and decreases area, by denoting $I(t)=\f{\abs{M_t}^{n+1}}{\abs{\Om_t}^n}$ as the isoperimetric ratio, we have the following

\begin{prop}\label{isovp}
    Along the flow \eqref{e-VPMCF}, $I_S(|\Omega_0|)\leq I(t)\leq I_S(|\Omega_0|)+\rho_0$.
\end{prop}

\begin{proof}
    The first inequality is precisely the isoperimetric inequality in the Schwarzschild space. The second inequality follows directly from the volume-preserving property of the flow. 
\end{proof}

A straightforward consequence of Proposition \ref{isovp} is that the area of the hypersurface is bounded from below, 
\[
\abs{M_t}=I(t)^{\f1{n+1}}\abs{\Om_t}^{\f n{n+1}}\geq(I_S(|\Omega_0|)\abs{\Om_0}^n)^{\f1{n+1}}:=M_*.
\]
We now define the maximal time set as
\[
\mathcal{S}=\{\tau\in[0,T):\f14\leq\kappa_i\leq4, \quad\forall i=1, \cdots, n, \forall x\in M_t, \forall t\in[0,\tau)\}.
\]
Then, for all $ t\in\mathcal{S}$, there holds
\begin{enumerate}
    \item $\abs{H}(x,t)=\sum\kappa_i(x,t)\leq 4n\leq C'$;
    \item $\abs{A}(x,t)=\sqrt{\sum\kappa_i^2}\leq 4\sqrt{n}\leq C'$;
    \item $h(t)=\frac{\int_{M_t}Hd\mu_t}{\int_{M_t}d\mu_t}\leq C'$.
\end{enumerate}

Let $S'=\sup\mathcal{S}$ and assume $S'<\infty$. 
We claim that the maximal time is bounded from below by a positive dimension-dependent constant, which implies that the flow exists for some fixed short time interval. 

\begin{thm}
    There exists some positive constant $C=C_n$ such that $S'>C_n$.
\end{thm}

\begin{proof}
It suffices to establish an upper bound for the maximal principal curvature $\kappa_{max}=\max_{M_t}\kappa_i$ and a lower bound for the minimal principal curvature $\kappa_{min}=\min_{M_t}\kappa_i$. 
Since $\kappa_{max},\kappa_{min}$ are not differentiable, we employ an approximation procedure. 
For any $\beta>0$, define: 
\begin{align*}
u_{2}(x_{1},x_{2}) &=\quad\frac{x_1+x_2}2+\sqrt{\left(\frac{x_1-x_2}2\right)^2+\beta^2},\\
u=u_{n}(x_1,\ldots,x_{n})&=\quad\frac1{n}\sum_{i=1}^{n}u_2(x_i,u_{n-1}(x_1,\ldots,\hat{x}_i.\ldots,x_{n})),n\geq3.
\end{align*}
Direct calculations yield 
\[
    \pfrac{}{t}u=\f {\pa u}{\pa h_{ij}}(\pfrac{}{t}h_{ij}),\]
\[\Delta u=\f{\pa^2u}{\pa h_{pq}\pa h_{ij}}\n^lh_{pq}\n_lh_{ij}+\f{\pa u}{\pa h_{ij}}\Delta h_{ij}.
\]
Given that the curvatures are bounded, specifically, $\abs{\bar{R}m},\abs{\bar{\n} \bar{R}m}\leq \Lambda$, by the properties of $u$ (see Lemma 3.4 in \cite{Miglioranza2020TheVP}) and the evolution equations in Lemma \ref{evoeqns}, we derive
\begin{align*}
    \pfrac{}{t}u 
    =& \Delta u+\f{\pa u}{\pa h_{ij}}\left(\pfrac{}{t}h_{ij}-\Delta h_{ij}\right)-\f{\pa^2u}{\pa h_{pq}\pa h_{ij}}\n^lh_{pq}\n_lh_{ij}\\
    \leq &\Delta u + \f {\pa u}{\pa h_{ij}}\left(\abs{A}^2h_{ij}+\Lambda (5h_{ij}+\delta_{ij}h)-hh_{il}h_{lj}+2\Lambda\right)\notag\\
    &-\f{\pa^2u}{\pa h_{pq}\pa h_{ij}}\n^lh_{pq}\n_lh_{ij}\\
     \leq &\Delta u+2\abs{A}^2u+5\Lambda u+2\Lambda+\Lambda h\\
    \leq& \Delta u+(2C'^2+5\Lambda)u+2C'\Lambda.
\end{align*}
The maximum principal then implies that 
\[
u(t)\leq u(0)e^{(2C'^2+5\Lambda)t}+\frac{2C'\Lambda}{2C'^2+5\Lambda}(e^{(2C'^2+5\Lambda)t}-1).
\]
Letting $\beta\rightarrow0$, we have $u(t)\rightarrow \kappa_{max}(t)$. By solving the inequality
\[ 2e^{(2C'^2+5\Lambda)t}+\frac{2C'\Lambda}{2C'^2+5\Lambda}(e^{(2C'^2+5\Lambda)t}-1)\leq4, \]
we get there is $T_1=T_1(n)$ such that $\kappa_{max}(t)\leq 4$ when $t\leq T_1$. We can similarly prove that for some $T_2=T_2(n)$, $\kappa_{min}(t)\geq\f14$ when $t\leq T_2$. Therefore, we obtain
\[ S'\geq min\{T_1, T_2\}:=C_n. \]
\end{proof}

Once the maximal time lower bound is established, we can use it to control the geometry at this maximal time, which in turn helps to refine the curvature estimates. Hence, we find ourselves in a situation where the conditions match the initial setup exactly. We can then extend the flow a bit longer, which leads to a contradiction. To control the geometry of the hypersurface at the maximal time, we need to estimate the derivatives of the second fundamental form.
The following estimate is Theorem 1.14 in \cite{Miglioranza2020TheVP} and we state it briefly.

\begin{prop}\lab{thm3.2}
    Assume $\abs{H}(x,t)\leq C_1,\abs{A}(x,t)\leq C, \forall t\in[0,T']$. Suppose further that $\abs{\bar{R}m}\leq C_2, \abs{\bar{\n}\bar{R}m}\leq C_2$. Then there exists a constant $D=D(C,C_1,C_2,T,M_0)$ such that, for $\forall t\in (0, T']$, 
    \[
    \sup_{M_t}t\abs{\n A}^2\leq D.
    \]
\end{prop}
\begin{proof}
   Recall the evolution equation for $\abs{\n A}^2$:
    \[
    \begin{aligned}
    \frac{\partial}{\partial t}|\nabla A|^{2}&= \Delta|\nabla A|^2-2|\nabla^2A|^2+A*A*\nabla A*\nabla A+\nabla A*\nabla A*\overline{Rm}+\\
    &+A*\nabla A*\bar{\nabla}\overline{Rm}+\nabla A*\bar{\nabla}^2\overline{Rm}+hA*\nabla A*\nabla A+\\
    &+h\nabla A*\bar{\nabla}\overline{Rm}+hA*\nabla A*\overline{Rm}.
    \end{aligned}
    \]
   Define the  function $f(x, t)$ as follows: for some positive constants $\Lambda,L_1$ and $L_2$ that depend only on the given constants $C_1,C_2,C,T$,
$$
f(x,t)=t\:|\nabla A|^2\big(\Lambda+L_1|A|^2\big)+\frac{L_2}{2}|A|^2.
$$
After a lengthy and tedious computation, we finally obtain the following estimate, where the constant $L, R, K_5$ depends again only on the initial geometry and $C_1, C_2, C, T$:
\[
\begin{aligned}
\Big(\frac{\partial}{\partial t}-\Delta\Big)f(x,t)
\leq Lf(x,t)+tR+K_{5}.
\end{aligned}
\]
The maximum principle again gives that 
\[
\sup_{x\in M_t}|\nabla A|\leq\frac{D}{\sqrt{t}},\quad \forall t\in(0,T'].
\]
\end{proof}

We now restate Theorem \ref{thm1.3} here for readers' convenience and apply this to show the long-time existence. 

\begin{thm}\lab{thm3.4}
    Let  $\Om_k\Subset \Om\sub N$ be a sequence of bounded domains in the Schwarzschild space, where $\Omega$ contains the horizon, and let  $M_k\subset\pa \Om_k$ be a sequence of smooth boundary components homologous to the horizon (in particular, $\pa\Om_k$ is an embedding). Let $\Om_0$ denote the domain enclosed by the horizon and some coordinate sphere such that there exist two positive constants $C_i, i=1,2$, and a sequence of $\rho_k$ for all $k$, satisfying
    \begin{enumerate}
        \item $\abs{\Om_k}=\abs{\Om_0}, \abs{\pa\Om_k}\leq C_1$;
        \item $\abs{A}_{\pa\Om_k}\leq C_2, \abs{\n A}_{\pa\Om_k}\leq C_2$;
        \item $\f{\abs{M_k}^{n+1}}{|{\Om_k}|^n}\leq I_S(|\Omega_0|)+\rho_k$, with $\rho_k\rightarrow0$ as $k\to \infty$;
    \end{enumerate}
    where $|\cdot|$ represents the volume of a domain or the area of a hypersurface, $A$ is the second fundamental form, and $I_S$ is the isoperimetric ratio of the Schwarzschild space defined in
    Theorem \ref{thm1.3}.
Then, up to a subsequence, $M_k$ converges to some coordinate sphere $\Sgm_\infty$. Let $\Om_\infty$ denote the domain enclosed by $M_\infty$ and the horizon; we have $\abs{\Om_\infty}=\abs{\Om_0}$. 
\end{thm}

    The proof is modelled from the compactness theory for the minimal surfaces \cite{colding2011course,Miglioranza2020TheVP}. Here, we would like to emphasize that the ambient space is not the Euclidean space. 
    Thus, additional effort is required to estimate the graph function and its derivatives. We aim to illustrate that a properly embedded hypersurface $M$ with bounded norm of the second fundamental form must admit a uniform radius, meaning $M$ can be represented as a graph over some ball of the fixed radius \cite{Ros2002global}.
    To show this, we first introduce some notations. Fix a point $p\in M$ and a radius $r>0$ and denote by 
		\[
		B(p,r)=\{p+v: v\in T_pM, \abs{v}<r\},
		\]
		and 
		\[
		W(p,r)=\{q+t\nu(q): q\in B(p,r), t\in\R, \nu \text{ is the Gauss map}\},
		\]
		\[
		W(p,r,\eps)=\{q+t\nu(q): q\in B(p,r),\abs{t}<\eps\}.
		\]
        
\begin{lem}[Uniform Graph Lemma for manifolds]\lab{lem3.5}
Let $M\sub N$ be a properly embedded hypersurface. Suppose there exists a positive constant $c$ such that $\abs{A}\leq c$ on $M$. Let $p\in M$ be an arbitrary point. With the abuse of notation, we can find a positive constant $R=R(c,N)$ and a smooth function $u:B(p, R)\rightarrow\R$ such that
\begin{enumerate}
    \item $M\cap W(p,R)=graph(u)\cap W(p,R)$.
    \item $\abs{u}, \abs{\n u}, \abs{\n^2 u}$ are all bounded.
\end{enumerate}
\end{lem}

\begin{proof}
    Since $M$ is an embedding, locally we can find a cylinder $W(p,R)$ and a smooth function $u$ with the property (1) holds. We now want to show that we can find a uniform lower bound of $R$ depending only on the geometry of the ambient space and the constant $c$. Choose normal coordinates $\{x_i\}$ around $p\in M$;  locally up to a rotation, we can assume $p=0,T_pM=\set{y=0}$. The graph around $p$ is given by $y=\psi(x)=x+u(x)\nu(p)$, and $\nu(p)=(0,\cdots, 1)$, where $u:B(p,R)\sub T_pM\rightarrow \R$ and $\nu$ is the outward unit normal of $M$ in $N$. 
     We define 
    \[
    W^2=1+\bar{g}^{ij}u_iu_j.
    \] 
    Let $\{e_i\}$, $i=1, 2,\cdots,n+1,$ denote the orthonormal basis in $N$. And set $\nu_{n+1}=\<\nu,e_{n+1}\>=\f1W$. It is obvious that $\nu_{n+1}(p)=1$, so we may assume $\nu_{n+1}\geq\f12$ in $B(p,R)$.  
    \begin{claim}
     $\abs{\n \nu_{n+1}}\leq 2c$.
    \end{claim}

Proof of claim,
    \begin{align*}
        &(\nu_{n+1})_{x_i}  =  \<\n_{\psi_{x_i}}\nu, e_{n+1}\>\\
         = & \<\sum_jA(\psi_{x_i},\psi_{x_j})\psi_{x_j}, e_{n+1}\>
         =  A_{ij}u_j.
    \end{align*}
So we have 
\[
\abs{\n \nu_{n+1}}^2\leq \sum_i(\sum_jA_{ij}u_j)^2\leq W^2|A|^2=\nu_{n+1}^{-2}|A|^2\leq4c^2.
\]

Next we aim to derive a lower bound for the radius. Suppose now that $R$ is the maximal radius for which properties (1) and (2) hold. Observe that if $u$ is defined on $\pa B(p,R)$ with $\nu_{n+1}>\f12$ on $\pa B(p,R)$, then we can continuously extend $u$ beyond $B(p,R)$, which contradicts the maximality of $R$. It follows that two cases will happen. 
\begin{enumerate}
    \item[a)] There must exist some $q\in\pa B(p,R)$ with $\nu_{n+1}(q)=\f12$;
    \item[b)] There exists a sequence $ \{q_k\}\in B(p,R)$ with  $d(\psi(q_k),\pa N)\rightarrow0$.
\end{enumerate}
In case a), we obtain
\[
\f12=\abs{\nu_{n+1}(p)-\nu_{n+1}(q)}\leq \abs{\n \nu_{n+1}}|p-q|\leq 2cR.
\]
Thus, we have $R\geq \f1 {4c}$. In case b), 
\[
d(p,\psi(q_k))\leq length(\psi[p,q_k]),
\]
where $[p, q_k]$ is the segment in $B(p, R)$ joining $p$ and $q_k$.
However, 
\[
length(\psi[p,q_k])=\int_0^{\abs{q_k}}\sqrt{1+\abs{\n u}^2}ds<2\abs{q_k}<2R,
\]
yielding 
\[
d(p,\pa N)\leq d(p, \psi(q_k))+d(\psi(q_k),\pa N)<2R+d(\psi(q_k),\pa N)\to 2R.
\]
Therefore, setting $R=\min\set{\f1{4c},\f12d(p,\pa N)}$, we obtain a uniform lower bound for the radius, which proves (1) in Lemma \ref{lem3.5}.
Concerning (2), 
define $E_i=\psi_{x_i}$ and linear maps $T_i^j$ by $E_i=e_i+u_ie_{n+1}=T_i^je_j$; then $\{E_i\}$ forms a basis for $T_pM$. Let $g_{ij}=\bar{g}(E_i,E_j)$ denote the induced metric on $M$, and let $\Gamma_{ij}^k$ denote the Christoffel symbols of $N$.
     The unit normal is given by $\nu=\f1W( -\nabla u,1)$ and
    \[
    \<\nu,e_i\>=-\f{u_i}W,\quad \<\nu,e_{n+1}\>=\f1W.
    \]
    The mean curvature of $M$ is given by 
  \[
    H=-g^{ij}\<\nu,\n_{E_i}E_j\>.
   \]
    We compute that 
 \[
    g^{ij}=\bar{g}^{ij}-W^{-2}\bar{g}^{ip}\bar{g}^{jl}u_pu_l,
 \]
and 
\begin{align}\lab{eqn3.4}
    \n_{E_i}E_j &=(T_i^p\bar{\n}_{e_p}T_j^le_l)^T
    =\left(T_i^pT_j^l\bar{\n}_{e_p}e_l+(T_i^p\bar{\n}_{e_p}T_j^l)e_l\right)^T\\ \notag
    & = T_i^pT_j^l\Gamma_{pl}^m e_m+E_i(T_j^l)e_l    =T_i^pT_j^l\Gamma_{pl}^me_m+u_{ij}e_{n+1}. 
\end{align}
Thus, we obtain 
\[
A_{ij}=-\<\nu,\n_{E_i}E_j\>=\f{u_mT_i^kT_j^l\Gamma_{kl}^m}W-\f{u_{ij}+T_i^kT_j^l\Gamma_{kl}^{n+1}}W.
\]
Substituting $T_i^p=\delta_i^p+u_i\delta_{n+1}^p$ into equation \rf{eqn3.4}, we obtain 
\begin{align}\label{ainj}
    A_{ij}= & \f{u_m}W (\Gamma_{ij}^m+u_j\Gamma_{in+1}^m+u_i\Gamma_{n+1j}^m+u_iu_j\Gamma_{n+1n+1}^m)\\ \notag
    - & \f1W(u_{ij}+\Gamma_{ij}^{n+1}+u_j\Gamma_{in+1}^{n+1}+u_i\Gamma_{n+1j}^{n+1}+u_iu_j\Gamma_{n+1n+1}^{n+1}).
\end{align}
 
Noting that $M$ is locally a graph over some balls of fixed radius and given our assumption that $W\leq2$, it follows from $\abs{A}\leq c$ that  $\abs{\n^2 u}\leq 2c+28C(n)$, where $C(n)$ is an upper bound of $\Gamma_{ij}^k$ in \eqref{ainj}. Finally using the mean value theorem, we obtain
\[ |u(q)|=|u(p)-u(q)|\leq|\n u||p-q|\leq |\n^2 u||p-q|^2, \]
which completes the proof.
 \end{proof}

\begin{proof}[Proof of Theorem \ref{thm3.4}]
    Using the uniform graph lemma, we can find a uniform radius $R>0$ such that, on $B(p,R)$, every $M_k\cap W(p,R)$ can be represented as disjoint graphs $\cup_l(V_k^l\cap B(p,R))$, for $l=1,\cdots,s_{k,p}=s(p,k),$ with graph function $v_k^l:B(p,R)\rightarrow \R$. 
    Furthermore, the derivatives of $v_k^l$ up to second order are all bounded. 
    Hence, by passing to a subsequence, the graph functions $v_k^l$  converge to a limiting function $v^l:B(p,R)\rightarrow\R$. This shows that $M_k$ converges  to limiting graphs $V^l$ in $W(p,R)$.
    
    Now, the area of the graphs are uniformly bounded. Indeed, on each graph $V_k^l$, $(g^l_k)_{ij}=\bar{g}_{ij}+(v_k^l)_i(v_k^l)_j$, we immediately get $\det((g_k^l)_{ij})=\det(\bar{g}_{ij})(1+U_k^l\bar{g}^{-1}{U_k^l}^T)$, where $U_k^l=((v_k^l)_1,\cdots,(v_k^l)_n)$. It follows that 
    \[
Area(graph(v_k^l))=\int_{B(p,R)}\sqrt{\det((g_k^l)_{ij})}=\int_{B(p,R)}W_k^l\sqrt{\det(\bar{g}_{ij})}
    \]
   Since $W_k^l$ is uniformly bounded, we then obtain that 
   \[
   2\mu_N(B(p,R))\geq Area(graph(v_k^l))\geq \mu_N(B(p,R)).
   \]
   for every $k$ and every $l$, where $\mu_N$ is the volume measure of $N$. Therefore, from $M_k\cap W(p,R)=\cup_l(V_k^l\cap B(p,R)$, we get that 
   \[
   s_{k,p}\leq \f{\mu_N(M_k\cap W(p,R))}{Area(graph(v_k^l))},
   \]
    this shows  
    \[
    s_{k,p}\leq \f{\mu_N(M_k)}{c_{N,R}\om_nR^{n+1}}\leq C_1,
    \]
    where $c_{N,R}$ is a positive number depending on the geometry of $N$ and radius $R$, since the ambient space is curved and $C_1=C_1(N,\max{\abs{A}},d(p,\pa N))$. In particular, the multiplicity of the limiting surface is finite.
    
    Since the third condition ensures that the isoperimetric ratio of $M_k$ converges to $I_S(|\Omega_0|)$, this implies that the multiplicity $s=1$.
    By the isoperimetric inequality,  the limiting hypersurface has to be a coordinate sphere since this limiting hypersurface attain the isoperimetric ratio, this completes the proof.
\end{proof}

As an application, we have the following corollary.

\begin{cor}\label{cor3.6}
Let $(N^{n+1}, \bar{g})$ be the Schwarzschild space. For any $\delta>0$ and hypersurface $M$, homologous to the horizon, there exists $$\rho=\rho(\norm{A_M},\norm{\n A_M},Area(M),\delta,n)>0$$ such that if 
the isoperimetric profile of $M$ is controlled by $I_{M}\leq I_S(|\Omega|)+\rho$, where $I_M=\f{\abs{M}^{n+1}}{\abs{\Omega}^n}$, and $\Omega$ denotes the domain enclosed by $M$ and the horizon $\Sigma_0$, and $\norm{A_M},\norm{\n A_M}$ are all controlled, then $M$ must be $C^2$-close to some coordinate sphere $\Sigma_r$. 
Moreover, there exists a smooth function $u:\Sigma_r\rightarrow M$ such that $M$ can be represented as a radial graph over the coordinate sphere $\Sigma_r$ with $\norm{u}_{C^2(\Sigma_r)}\leq\delta$.
\end{cor}

In view of Proposition \ref{thm3.2}, it follows that $\abs{\n A}$ is uniformly bounded at time $t=S'$. 
Moreover, the derivatives of the second fundamental form can be bounded up to any order. 
Thus, by the Arzela-Ascoli theorem, we can in fact take the limit as $t\rightarrow S'$. 
This convergence is smooth due to the uniform bounds of $\abs{A}^2$ and $\abs{\n^kA}^2$ for all $k$, as guaranteed by Proposition \ref{thm3.2}. 
Hence, we obtain a smooth hypersurface $M_{S'}$ at the maximal time $t=S'$, which satisfies the same conditions in Theorem \ref{thm1.1} as the initial hypersurface. 
Applying Corollary \ref{cor3.6}, we then  obtain that the range of $\kappa_i(x,S')$ can be refined to $[\f12,2]$. We can repeat the argument  with $M_{S'}$ taken as the initial hypersurface. 
Consequently, the flow exists for $t\in[0, +\infty)$. The proof of long time existence is completed.

\subsection{The convergence of the VPMCF}

This subsection is devoted to establishing the convergence of the volume-preserving mean curvature flow, which essentially relies on the uniqueness of the isoperimetric hypersurface in the Schwarzschild space. We start by analyzing the asymptotic behavior of the mean curvature $H$. 

\begin{lem}\lab{lem3.4}
    The mean curvature $H$ converges uniformly to a constant, i.e. 
    \[
    \lim_{t\rightarrow\infty}\max_{M_t}\abs{H(x,t)-h_t}=0.
    \]
\end{lem}
\begin{proof}
We want to find a uniform bound for the mean curvature $H$ in a space-time neighborhood of a point $(p_1, t_1)\in M\times[0,\infty)$. As in the Euclidean case, since 
\[
\n H(x,t)=g^{ij}\n A_{ij}(x,t),
\]
the uniform bounds on the covariant derivatives of the second fundamental form yield  uniform control on $\n H(x,t)$, making $H$ spatially Lipschitz continuous. 
Recall the evolution equation for $H$ in Lemma \ref{evoeqns},
\[
\f{\pa H}{\pa t}=\Delta H+(h-H)\left(\abs{A}^2+\overline{Ric}(\nu,\nu)\right),
\]
which implies that 
\begin{align}
    \abs{\pd t H} & \leq \abs{\Delta H}+(\abs{h}+\abs{H})\abs{A}^2+\abs{h-H}\abs{\overline{Ric}}.
\end{align}
Therefore, $\abs{\pd t H}$ is bounded because of the uniform bound on $\abs{\n^2A}$. 
To summarize, for some $D>0$, we have the uniform estimate 
\[
\abs{\n H}, \abs{\pa_tH}\leq D.
\]
We now recall that
\[
-\pd t\abs{M_t}=\int_{M_t}(H-h)^2d\mu_t>0.
\]
It follows that for any $\eta>0$, we have
\[
\int_{M_t}(H-h)^2d\mu_t\geq\eta\Longleftrightarrow  \pd t\abs{M_t}\leq-\eta.
\]

For this reason, if at some point $(p_1, t_1)$ in the space-time such that $\abs{H-h}=c$ for some $c>0$, then it remains larger than $c/2$ on a space-time neighbourhood of $(p_1, t_1)$ with radius $r$, where  $r=r(c)$ is uniform. 
Taking into account also the bounds for $\abs{M_t}$, we deduce that $\partial_t \abs{M_t}<-\eta$, for $t\in[t_1-r,t_1+r]$ and some $\eta=\eta(c)$. 
Since $\abs{M_t}$ is monotonically decreasing and bounded below by $\abs{M_t}\geq M^*$, this can happen only for a finite number of intervals for any given $c>0$. This shows that $\abs{H-h}$ tends to zero uniformly, i.e.
\[
\lim_{t\rightarrow\infty}\max_{x\in M_t}\abs{H(x,t)-h(t)}=0.
\]
\end{proof}

To see the full time convergence, we note that the constant mean curvature (CMC) hypersurfaces in the Schwarzschild space with fixed volume are unique. 
This uniqueness implies the full time convergence. 
Suppose otherwise; we would have two sub-sequences, say $M^1_i$ and $M^2_j$, which by our assumptions must converge to two distinct limiting hypersurfaces $M^1_\infty$ and $M^2_\infty$. 
However, $M^1_\infty$ and $M^2_\infty$ are both constant mean curvature hypersurfaces with the same volume, and by Lemma \ref{lem3.4}, their mean curvatures are identical. 
Thus, they must coincide, which is a contradiction. 
This finishes the proof of Theorem \ref{thm1.1}.

\subsection{The long time existence and convergence of the APMCF} In this subsection, we will prove Theorem \ref{thm1.1-2}. Following the same argument as we did for the VPMCF by an iteration argument, we see that if $M_0$ is $\delta$-close in $C^2$-norm to a coordinate sphere $\Sigma_{r_0}$ with the area of $M_0$ equals the area of $\Sigma_{r_0}$, then the APMCF exists for all time and subconverges to some limiting hypersurface. The only difference is the convergence of the flow. We first prove that:

\begin{lem}\lab{lem3.7}
   Along the APMCF, the mean curvature $H$ converges to some constant uniformly, i.e. 
    \[
    \lim_{t\rightarrow\infty}\max_{M_t}\abs{H(x,t)-h_0(t)}=0.
    \]
\end{lem}

\begin{proof} By (\ref{e-APMCF-vol-2}), we see that
    \begin{equation*}
        \frac{d}{dt}{\rm Vol}(\Omega_t)=\int_{M_t}(h_0-H)d\mu_t=\frac{\int_{M_t}H^2d\mu_t}{\int_{M_t}Hd\mu_t}|M_t|-\int_{M_t}Hd\mu_t.
\end{equation*}
Since $|\nabla^k A|\leq C(k)$ and ${\rm Area}(M_t)={\rm Area}(M_0)$, we see that 
    \begin{equation}\label{e-3.10}
        \left|\frac{d^2}{dt^2}{\rm Vol}(\Omega_t)\right|\leq \Lambda,
\end{equation}
for some constant $\Lambda$ independent of $t$. We will show that
    \begin{equation*}
        \lim_{t\to\infty}\frac{d}{dt}{\rm Vol}(\Omega_t)=0.
\end{equation*}
It suffices to show that
    \begin{equation*}
        \limsup_{t\to\infty}\frac{d}{dt}{\rm Vol}(\Omega_t)=0.
\end{equation*}
Suppose not, then there will be a sequence $t_i$ tending to infinity and a positive constant $c>0$ such that
    \begin{equation*}
        \frac{d}{dt}{\rm Vol}(\Omega_t)(t_i)\geq c>0.
\end{equation*}
By (\ref{e-3.10}), we see that if we set $\tau=\frac{c}{2\Lambda}$, then for $t\in [t_i-\tau, t_i+\tau]$
    \begin{equation*}
        \frac{d}{dt}{\rm Vol}(\Omega_t)(t)\geq \frac{c}{2}>0.
\end{equation*}
Since $t_i$ tends to infinity, up to a subsequence, we may assume that $t_i+\tau<t_{i+1}-\tau$ for each $i$. Now using the fact that the volume is nondecreasing along the APMCF, we compute
\begin{align*}
    {\rm Vol}(\Omega_{t_i+\tau})
    &\geq {\rm Vol}(\Omega_{t_i-\tau})+c\tau\\
    &\geq {\rm Vol}(\Omega_{t_{i-1}+\tau})+c\tau\\
    &\geq {\rm Vol}(\Omega_{t_{i-1}-\tau})+2c\tau\\
    &\geq\cdots\\
    &\geq {\rm Vol}(\Omega_{t_0+\tau})+ic\tau\to\infty,
\end{align*}
as $i\to\infty$. This contradicts the fact that
    \begin{equation*}
       {\rm Vol}(\Omega_t)\leq {\rm Area}(M_t)^{\frac{n+1}{n}}I_S^{\frac{1}{n+1}}={\rm Area}(M_0)^{\frac{n+1}{n}}I_S^{\frac{1}{n+1}}.
\end{equation*}
Therefore, we must have
    \begin{equation*}
        \lim_{t\to\infty}\frac{d}{dt}{\rm Vol}(\Omega_t)=\lim_{t\to\infty}\int_{M_t}(h_0-H)d\mu_t=0.
\end{equation*}
Then the lemma follows by standard elliptic estimates.
\end{proof}

Having already derived the estimates for all derivatives of the second fundamental form, we conclude that a subsequence of the flow converges to a constant mean curvature (CMC) hypersurface. Now the full convergence of the flow follows via precisely the same argument as for the VPMCF using the uniqueness of the CMC hypersurfaces in the Schwarzschild space. This completes the proof of Theorem \ref{thm1.1-2}.

\vspace{.2in}

\section{Stability of volume and area preserving mean flow in asymptotically Schwarzschild spaces}

In this section, we aim to generalize the previous results of the corresponding volume and area preserving flows to asymptotically Schwarzschild spaces under different initial conditions. Precisely, we will show that if an initial hypersurface is sufficiently close to  $\Sigma_V$, the isoperimetric hypersurface constructed by Eichmair and Metzger (\cite{EichmairInventMath2013}), then the flow will exist long time and converge to a limiting hypersurface exponentially fast. We begin by recalling the definition of asymptotically Schwarzschild spaces. We use the same notation $(N^{n+1},\bar{g})$ to denote the ambient space.
\begin{defn}
    A connected complete Riemannian manifold $(N^{n+1},\bar{g})$ is called to be $C^k$-close to Schwarzschild space, if there exists a bounded open set $U\subset M$ so that $M\setminus U$  is diffeomorphic to $\R^{n+1}\setminus B_1(0)$ and in the coordinates induced by $x = (x_1,x_2,\cdots,x_{n+1})$ we have
 that, for $m>0$ and $k\geq0$ is an integer,
 \[
 \sum_{l=0}^kr^{2+l}\abs{\pa^l(\bar{g}-g_m)_{ij}}\leq C,
 \]
 where $(g_m)_{ij} = (1 + \f m{2r^{n-1}})^{\f4{n-1}}\delta_{ij}$, $r=\sqrt{ x^2_1 +\cdots+x^2_{n+1}}$.
\end{defn}

Eichmair and Metzger (\cite{EichmairInventMath2013}) proved that if $(N^{n+1},\bar{g})$ is $C^2$-close to Schwarzschild space of mass $m>0$, then there exists $V_0>0$ such that for all $V\geq V_0$, there is a unique isoperimetric hypersurface $\Sigma_V$ with enclosed volume $V$. Furthermore, it is a strictly stable CMC hypersurface, and $C^2$-close to a coordinate sphere. We will use the following Corollary, whose proof is similar to Theorem \ref{thm3.4} and Corollary \ref{cor3.6}:

\begin{cor}\label{cor4.2}
Let $(N^{n+1}, \bar{g})$ be $C^2$-close to Schwarzschild space. For any $\delta_0>0$ and hypersurface $M$, homologous to the horizon, there exists $$\rho=\rho(\norm{A_M},\norm{\n A_M},Area(M),\delta_0,n)>0$$ such that if 
the isoperimetric profile of $M$ satisfies $I_{M}\leq I_{\Sigma_V}+\rho$, where $I_M=\f{\abs{M}^{n+1}}{\abs{\Omega}^n}$, $\Omega$ denotes the domain enclosed by $M$ and the horizon $\Sigma_0$ with $|\Omega|=V$, and $\norm{A_M},\norm{\n A_M}$ are all bounded, then $M$ must be $C^2$-close to $\Sigma_V$. 
Moreover, there exists a smooth function $u:\Sigma_V\rightarrow M$ such that $M$ can be represented as a radial graph over $\Sigma_V$ with $\norm{u}_{C^2(\Sigma_V)}\leq\delta_0$.
\end{cor}

For convenience, we restate the existence and convergence theorem in asymptotically Schwarzschild manifolds as follows:

\begin{thm}
    Let $(N^{n+1}, \bar{g})$ be a $C^2$ asymptotically Schwarzschild space. There exists a constant $\delta > 0$ such that the following holds: if a hypersurface $M_0$ satisfies two conditions:
    \begin{itemize}
        \item  the volume of the region enclosed by $M_0$ and the horizon satisfies $V \geq V_0$;
        \item $M_0$ is $\delta$-close to $\Sigma_V$ in the $C^2$ norm,
         \end{itemize}
    then the volume- and area-preserving mean curvature flow exists for all time and converges to $\Sigma_V$ as $t \to \infty$.
\end{thm}

\begin{proof} 
First, note that since $\Sigma_V$ is a strictly stable CMC hypersurface, there exists a constant $\delta_1 > 0$ such that no other CMC hypersurfaces exist within the $\delta_1$-neighborhood of $\Sigma_V$ (in the $C^2$ norm). The assumption implies the initial hypersurface $M_0$ satisfies $\frac{1}{2}\underline{\kappa} \leq \kappa_i \leq 2\overline{\kappa}$, where $\kappa_i$ denotes the principal curvatures of $M_0$, and $\underline{\kappa}$, $\overline{\kappa}$ are the lower and upper bounds of the principal curvatures of $\Sigma_V$, respectively. We also have $I_{M_0} < I_{\Sigma_V} + \rho$ for some sufficiently small $\rho$. The subsequent argument follows similarly to that in the Schwarzschild space case.

To elaborate, we first define the maximal existence time of the flows. If this time is finite, we can use curvature estimates to prove the flow extends beyond this maximal time. Note that the closeness of the isoperimetric ratio to that of the isoperimetric surface guarantees refined estimates for the second fundamental form. This is exactly what Eichmair and Metzger proved (see Theorem 9 in \cite{EichmairLarge2013}), and thus implies an extension property. This extension property, in turn, ensures the long-time existence of the flows.  

To establish subconvergence, we employ a similar strategy to show that the velocity — specifically, $(H - h)$ (for the volume-preserving case) and $(H - h_0)$ (for the area-preserving case) — will go to zero as $t \to \infty$. This follows from the monotonicity of both flows; then the pointwise limit is derived via elliptic estimates. If $\delta < \delta_1$ is sufficiently small, it is straightforward to see by Corollary \ref{cor4.2} that $M_t$ remains within the $\delta_1$-neighborhood of $\Sigma_V$. Finally, the convergence of the entire flow follows from the uniqueness of CMC hypersurfaces in this neighborhood.

Next we will use the center manifolds analysis to show that the entire flow converges to the limiting hypersurface exponentially fast \cite{Analytic1995}.
    To see this, we write the hypersurface as a graph over some coordinate sphere $\Sigma_0$, that is, 
    \[
    \Sigma_t=\set{x+u(x,t)\nu_0(x)|x\in \Sigma_0,\nu_0 \textrm{ is the outward unit normal of }\Sigma_0}.
    \]
    If we write the embedding as 
    \[
    \Phi(x)=x+u(x)\nu_0(x),
    \] 
then
\[\Phi_i(x)=e_i+u_i\nu_0(x)+u(x)\bar{\n}_{e_i}\nu_0=e_i+u_i\nu_0(x)+uA_{ik}e_k,
\]
where $e_i$ is the local frame in $\Sigma_0$ and $A_{ik}=\langle\bar{\nabla}_{e_i}\nu_0, e_k\rangle$. The induced metric is
    \[   g_{ij}=\sigma_{ij}+u_iu_j+u(\sigma_{il}A_{jl}+\sigma_{jk}A_{ik})+u^2A_{ik}A_{jl}\sigma_{kl},
    \]
    where $\sigma_{ij}$ is the induced metric on $\Sigma_0$. 
    The  normal vector of $\Sigma_t$ is 
    \bel{eqn4.2}
    \nu(\Phi(x))=(1+uH_0)\nu_0(x)-u_ie_i,
    \qe
    where $H_0=\sigma^{ij}A_{ij}$ is the mean curvature of $\Sigma_0$.
    As calculated in \rf{ainj}, we have the mean curvature expression, 
    \[
    H=-\f1Wg^{ij}u_{ij}+\mathrm{l.o.t},
    \]
    where $W^2=1+\bar{g}^{ij}u_iu_j$.
   Thus, we have the following evolution equation,
   \[
   \f{\pa u}{\pa t}\nu_0(x)=(h-H)\nu(x,t)=\f1W\left(\bar{g}^{ij}-\f1{W^2}\bar{g}^{ip}\bar{g}^{jq}u_pu_q\right)u_{ij}+\mathrm{l.o.t}.
   \]
   which is equivalent to 
   \[
  \f{\pa}{\pa t} u=(h-H)\<\nu_0(x),\nu(x,t)\>^{-1}.
   \]
   Combining the equation \rf{eqn4.2}, we have the following evolution equation
   \bel{eqn4.3}
   \f{\pa}{\pa t} u=\f{h_u-H_u}{1+uH_0}.
   \qe
   Here, we emphasize the dependence of the mean curvature and normalized mean curvature on the graph function $u$.
   We now apply the above calculation to our setting. Let $\Sigma_\infty$ denote the limiting hypersurface, which has constant mean curvature-this means that when calculating the variation in the evolution equation \rf{eqn4.3}, we only need to compute the term involving $(h-H)$ since $(h_\infty-H_\infty)=0$.
   Now, suppose we have a family of graph representations over the limiting hypersurface $\Sigma_\infty$
   \[
   \Sigma_\eps=\set{x+\eps \eta(x)\nu_\infty(x)|x\in\Sigma_\infty}.
   \]
   Let $V_\eps=\pa_\eps \phi_\eps=\eta\nu_\infty$ be the variation field. A direct computation gives that 
   \[
   \pa_\eps H_\eps|_{\eps=0}=-\Delta\eta-\left(\bar{R}ic(\nu_\infty,\nu_\infty)+\abs{A}^2\right)\eta,
   \]
   where $\Delta=\Delta_\Sigma$ is the Laplace-Beltrami operator on $\Sigma$. Also the variation of the volume form is given by 
   \[
   \pa_\eps d\mu_\eps=\f12tr_g\dot{g}=\textrm{div} V_\eps=\eta H.
   \]
   It follows that 
   \begin{align*}
   &\pa_\eps(h_\eps-H_\eps)|_{\eps=0}\\
   = &-\fint\left(\bar{R}ic(\nu_\infty,\nu_\infty)+\abs{A}\right)\eta d\mu+\fint H^2\eta d\mu
   + \Delta\eta+(\bar{R}ic(\nu_\infty,\nu_\infty)+\abs{A}^2)\eta\\
   = &\Delta\eta+(\bar{R}ic(\nu_\infty,\nu_\infty)+\abs{A}^2)\eta+\fint(H^2-\abs{A}^2-\bar{R}ic(\nu_\infty,\nu_\infty))\eta d\mu.
   \end{align*}
   Hence we obtain the following linearization of the evolution equation \rf{eqn4.3} at the limiting hypersurface
   \[
   \pa_t\eta=\Delta\eta+(\bar{R}ic(\nu_\infty,\nu_\infty)+\abs{A}^2)\eta+\fint(H^2-\abs{A}^2-\bar{R}ic(\nu_\infty,\nu_\infty)\eta d\mu.
   \]
   Denote by 
   \[
L\eta=\Delta\eta+\bar{R}ic(\nu_\infty,\nu_\infty)\eta+\fint(H^2-\bar{R}ic(\nu_\infty,\nu_\infty)-\abs{A}^2)\eta d\mu.
   \]
Now the full convergence follows by the argument in \cite{Volume2024,Analytic1995}, as the operator $L$ only has  negative eigenvalues. Specifically, since $M_{t_{i}}$ converges smoothly to the
limiting CMC hypersurface $\Sigma_\infty$ along a sequence  $t_{i} \to \infty$, for any sufficiently small $\varepsilon > 0$, there exists a sufficiently large time $t_{*}$  such that the oscillation of $u(\cdot, t_{i}) - u_\infty$ is less than $\varepsilon$ for all $t_{i} \geq t_{*}$. 
Write $M_{t}$ (for $t \geq t_{*}$) as the graph of the radial function $u(\cdot, t)$ over $\Sigma_\infty$. 
Since the oscillation of $u(\cdot,t_*)-u_\infty$ is already sufficiently small,  the argument in \cite{TheVolume1998}combined with Theorem 9.2.2 in \cite{Analytic1995} implies that the solution $u(\cdot,t)$ starting at $u(\cdot,t_*)$ exists globally in time and converges exponentially  to 0. 
This means that the hypersurface $\Sigma_t= \textrm{graph}u(\cdot,t)$ solves the flow equation  with initial condition $\bar{\Sigma}_{t_*}$.
By uniqueness, $\overline{\Sigma}_t$ coincides with $\Sigma_t$ for $t\geq t_*$; hence, the solution $M_t$ of (1.1) with initial condition $M_0$ converges exponentially  to the limiting CMC hypersurface $\Sigma_{\infty}$ as $t\to\infty$. 
\end{proof}

\textbf{Conflict of Interest} The authors have no conflict of interest to declare.

\textbf{Data availability} The authors declare no datasets were generated or analysed during the current study.

\bibliography{refofconv}

\begin{thebibliography}{10}

\bibitem{alikakos2003normalized}
Nicholas~D Alikakos and Alexandre Freire.
\newblock The normalized mean curvature flow for a small bubble in a riemannian manifold.
\newblock {\em Journal of Differential Geometry}, 64(2):247--303, 2003.

\bibitem{Brendle2013constant}
Simon Brendle.
\newblock Constant mean curvature surfaces in warped product manifolds.
\newblock {\em Publ. Math. Inst. Hautes \'{E}tudes Sci.}, 117:247--269, 2013.

\bibitem{colding2011course}
Tobias~H Colding and William~P Minicozzi.
\newblock {\em A course in minimal surfaces}, volume 121.
\newblock American Mathematical Soc., 2011.

\bibitem{EichmairLarge2013}
Michael Eichmair and Jan Metzger.
\newblock Large isoperimetric surfaces in initial data sets.
\newblock {\em J. Differential Geom.}, 94(1):159--186, 2013.

\bibitem{EichmairInventMath2013}
Michael Eichmair and Jan Metzger.
\newblock Unique isoperimetric foliations of asymptotically flat manifolds in all dimensions.
\newblock {\em Invent. Math.}, 194(3):591--630, 2013.

\bibitem{escher1998volume}
Joachim Escher and Gieri Simonett.
\newblock The volume preserving mean curvature flow near spheres.
\newblock {\em Proceedings of the American Mathematical Society}, 126(9):2789--2796, 1998.

\bibitem{TheVolume1998}
Joachim Escher and Gieri Simonett.
\newblock The volume preserving mean curvature flow near spheres.
\newblock {\em Proc. Amer. Math. Soc.}, 126(9):2789--2796, 1998.

\bibitem{MR840401}
M.~Gage and R.~S. Hamilton.
\newblock The heat equation shrinking convex plane curves.
\newblock {\em J. Differential Geom.}, 23(1):69--96, 1986.

\bibitem{MR848933}
Michael Gage.
\newblock On an area-preserving evolution equation for plane curves.
\newblock In {\em Nonlinear problems in geometry ({M}obile, {A}la., 1985)}, volume~51 of {\em Contemp. Math.}, pages 51--62. Amer. Math. Soc., Providence, RI, 1986.

\bibitem{MR906392}
Matthew~A. Grayson.
\newblock The heat equation shrinks embedded plane curves to round points.
\newblock {\em J. Differential Geom.}, 26(2):285--314, 1987.

\bibitem{Gui-Li-Sun}
Yaoting Gui, Yuqiao Li, and Jun Sun.
\newblock Stability of the area preserving mean curvature flow in asymptotic schwarzschild space.
\newblock {\em J. Funct. Anal.}, 289(7):Paper No. 111033, 26, 2025.

\bibitem{MR4910985}
Zheng Huang, Longzhi Lin, and Zhou Zhang.
\newblock Stability of the {V}olume {P}reserving {M}ean {C}urvature {F}low in {H}yperbolic {S}pace.
\newblock {\em Peking Math. J.}, 8(2):271--297, 2025.

\bibitem{MR0837523}
Gerhard Huisken.
\newblock Contracting convex hypersurfaces in {R}iemannian manifolds by their mean curvature.
\newblock {\em Invent. Math.}, 84(3):463--480, 1986.

\bibitem{MR0921165}
Gerhard Huisken.
\newblock The volume preserving mean curvature flow.
\newblock {\em J. Reine Angew. Math.}, 382:35--48, 1987.

\bibitem{huisken1996definition}
Gerhard Huisken and Shing-Tung Yau.
\newblock Definition of center of mass for isolated physical systems and unique foliations by stable spheres with constant mean curvature.
\newblock {\em Inventiones mathematicae}, 124(1-3):281--311, 1996.

\bibitem{alma991007096229706271}
O.A. (Olga~Aleksandrovna) Ladyzhenskaja, V.~A. (Vsevolod~Alekseevich) Solonnikov, and Nina~Nikolaevna Ural'ceva.
\newblock {\em Linear and quasi-linear equations of parabolic type}.
\newblock Translations of mathematical monographs ; v. 23. American Mathematical Society, Providence, 1968.

\bibitem{li2009volume}
Haozhao Li.
\newblock The volume-preserving mean curvature flow in euclidean space.
\newblock {\em Pacific journal of mathematics}, 243(2):331--355, 2009.

\bibitem{MR4762133}
Shiyang Li, Hongwei Xu, and Entao Zhao.
\newblock Volume preserving mean curvature flow of {$L^2$}-almost umbilical hypersurfaces in hyperbolic space.
\newblock {\em Internat. J. Math.}, 35(8):Paper No. 2450029, 15, 2024.

\bibitem{Analytic1995}
Alessandra Lunardi.
\newblock {\em Analytic semigroups and optimal regularity in parabolic problems}.
\newblock Modern Birkh\"{a}user Classics. Birkh\"{a}user/Springer Basel AG, Basel, 1995.
\newblock [2013 reprint of the 1995 original] [MR1329547].

\bibitem{MR2031450}
James~A. McCoy.
\newblock The mixed volume preserving mean curvature flow.
\newblock {\em Math. Z.}, 246(1-2):155--166, 2004.

\bibitem{Ros2002global}
W.~H. Meeks, III, A.~Ros, and H.~Rosenberg.
\newblock {\em The global theory of minimal surfaces in flat spaces}, volume 1775 of {\em Lecture Notes in Mathematics}.
\newblock Springer-Verlag, Berlin; Centro Internazionale Matematico Estivo (C.I.M.E.), Florence, 2002.

\bibitem{Miglioranza2020TheVP}
M~O Miglioranza.
\newblock The volume preserving mean curvature flow in a compact riemannian manifold.
\newblock {\em Theses}, 2020.

\bibitem{Volume2024}
Yong Wei, Bo~Yang, and Tailong Zhou.
\newblock Volume preserving {G}auss curvature flow of convex hypersurfaces in the hyperbolic space.
\newblock {\em Trans. Amer. Math. Soc.}, 377(4):2821--2854, 2024.

\bibitem{XuHW2014spaceform}
Hongwei Xu, Yan Leng, and Entao Zhao.
\newblock Volume-preserving mean curvature flow of hypersurfaces in space forms.
\newblock {\em Internat. J. Math.}, 25(3):1450021, 19, 2014.

\end{thebibliography}
\bibliographystyle{plain}

\date{November 2021}

\end{document}